\documentclass{scrartcl}
\usepackage{amsfonts}
\usepackage{amsmath}
\usepackage{amssymb}
\usepackage{hyperref}
\newcommand{\bbbn}{{\mathbb N}}
\newcommand{\bbbr}{{\mathbb R}}
\newcommand{\Idx}{{\mathcal I}}
\newcommand{\Jdx}{{\mathcal J}}
\newcommand{\ctI}{{\mathcal T}_{\Idx}}
\newcommand{\ctJ}{{\mathcal T}_{\Jdx}}
\newcommand{\lfI}{{\mathcal L}_{\Idx}}
\newcommand{\ctIJ}{{\mathcal T}_{\Idx\times\Jdx}}
\newcommand{\lfIJ}{{\mathcal L}_{\Idx\times\Jdx}}
\newcommand{\lfaIJ}{{\mathcal L}^+_{\Idx\times\Jdx}}
\newcommand{\lfiIJ}{{\mathcal L}^-_{\Idx\times\Jdx}}
\newcommand{\ctII}{{\mathcal T}_{\Idx\times\Idx}}
\newcommand{\lfII}{{\mathcal L}_{\Idx\times\Idx}}
\newcommand{\ctIII}{{\mathcal T}_{\Idx\times\Idx\times\Idx}}
\newcommand{\ctt}{{\mathcal T}_{\hat t_0}}
\newcommand{\cts}{{\mathcal T}_{\hat s_0}}
\newcommand{\ctts}{{\mathcal T}_{\hat t_0\times\hat s_0}}
\newcommand{\cttsr}{{\mathcal T}_{\hat t_0\times\hat s_0\times\hat r_0}}
\newcommand{\sons}{\mathop{\operatorname{sons}}\nolimits}
\newcommand{\treeroot}{\mathop{\operatorname{root}}\nolimits}
\newcommand{\pred}{\mathop{\operatorname{pred}}\nolimits}
\newcommand{\row}{\mathop{\operatorname{row}}\nolimits}

\allowdisplaybreaks
\newcommand{\qed}{}
\newtheorem{theorem}{Theorem}

\newtheorem{corollary}[theorem]{Corollary}
\newtheorem{definition}[theorem]{Definition}
\newtheorem{remark}[theorem]{Remark}
\newenvironment{proof}{\emph{Proof.}}{\strut\hfill $\Box$\medskip}
\title{Efficient arithmetic operations for rank-structured
       matrices based on hierarchical low-rank updates}
\author{Steffen B\"orm and Knut Reimer}
\date{\today}
\begin{document}
\maketitle
\begin{abstract}
Many matrices appearing in numerical me\-thods for partial differential
equations and integral equations are \emph{rank-structured}, i.e.,
they contain submatrices that can be approximated by matrices of
low rank.
A relatively general class of rank-structured matrices are
${\mathcal H}^2$-matrices: they can reach the optimal order of
complexity, but are still general enough for a large number of
practical applications.

We consider algorithms for performing algebraic operations
with ${\mathcal H}^2$-ma\-tri\-ces, i.e., for approximating the
matrix product, inverse or factorizations in almost linear
complexity.
The new approach is based on local low-rank updates that can be
performed in linear complexity.
These updates can be combined with a recursive procedure to
approximate the product of two ${\mathcal H}^2$-matrices, and
these products can be used to approximate the matrix inverse and
the LR or Cholesky factorization.

Numerical experiments indicate that the new algorithm leads to
preconditioners that require ${\mathcal O}(n)$ units of sto\-rage,
can be evaluated in ${\mathcal O}(n)$ operations, and take
${\mathcal O}(n \log n)$ operations to set up.
\end{abstract}

\thanks{\noindent Part of this research was funded by the Deutsche
  Forschungsgemeinschaft in the context of project BO 3289/4-1.}

\section{Introduction}
We consider an elliptic partial differential equation of the form
\begin{align*}
  -\mathop{\operatorname{div}} \sigma(x) \mathop{\operatorname{grad}}
  u(x) &= f(x) &
  &\text{ for all } x\in\Omega,\\
  u(x) &= 0 &
  &\text{ for all } x\in\partial\Omega,
\end{align*}
where $\Omega\subseteq\bbbr^d$ is a domain and
$\sigma:\Omega\to\bbbr^{d\times d}$ is uniformly symmetric positive
definite.

Using a Galerkin discretization with a finite element basis
$(\varphi_i)_{i\in\Idx}$ leads to a linear system
\begin{equation*}
  A x = b,
\end{equation*}
where the right-hand side $b\in\bbbr^\Idx$ corresponds to $f$ and
the solution $x\in\bbbr^\Idx$ represents the approximation of $u$.

In order to reach a sufficiently accurate approximation, it is
usually necessary to work with a large number $n=\#\Idx$ of
basis functions, therefore the linear system can become very large.

We are interested in constructing a good preconditioner
$B\in\bbbr^{\Idx\times\Idx}$, i.e., a matrix that can be evaluated
efficiently and that ensures that iterative solvers like the
preconditioned cg iteration \cite{HEST52} converge rapidly.

In this paper, we consider preconditioners based on rank-structured
matrices.
It has been proven \cite{BEHA03,BO07,FAMEPR13} that suitably chosen
submatrices $A^{-1}|_{\hat t\times\hat s}$, $\hat t,\hat s\subseteq\Idx$,
of the inverse of $A$ can be approximated by low-rank matrices.
The same holds for certain submatrices of the triangular factors $L$
and $R$ of the LR factorization $A = LR$ \cite{BE05,GRKRLE05a,FAMEPR13}.
These approximations have the remarkable property that they do not
depend on the smoothness of the coefficient function $\sigma$, even
discontinuous coefficients are permitted.
Similar results hold for the inverses of matrices resulting from the
Galerkin discretization of boundary integral operators \cite{FAMEPR13b}.

\emph{Hierarchical matrices} \cite{HA99,HAKH00,GRHA02,HA09} take
advantage of this property to construct approximations of $A^{-1}$ in
rank-structured representation:
a hierarchical decomposition of the index set $\Idx\times\Idx$
into appropriate subsets $\hat t\times\hat s$ is used to determine
submatrices that can be approximated by low-rank matrices, and
these matrices are represented in factorized form.
The result requires only ${\mathcal O}(n k \log n)$ units of storage,
where $k$ denotes the rank of the submatrices and depends on the
required accuracy.

In order to obtain an approximation of $A^{-1}$, hierarchical matrices
use block-wise constructions based on the multiplication of submatrices
and replace the exact matrix products by approximations that can be
computed efficiently.

A closer look at the theoretical results \cite{BO05} suggests that
the efficiency of the approximation can be improved by representing
the low-rank submatrices using hierarchically nested bases.
This leads to \emph{${\mathcal H}^2$-matrices}, originally developed
for integral operators \cite{HAKHSA00,BOHA02}.
Using an ${\mathcal H}^2$-matrix representation reduces the storage
requirements to ${\mathcal O}(n k)$.

Constructing a good ${\mathcal H}^2$-matrix preconditioner is
a challenging task.
So far, two approaches have been published:
if the nested bases are given a priori, the matrix multiplication can
be carried out in ${\mathcal O}(n k^2)$ operations \cite{BO04a},
and the inversion and the LR factorization can be reduced
to a sequence of matrix multiplications leading to a similar
complexity estimate.

In practice, the bases are typically not known in advance, so we have
to construct them during the course of our algorithm.
The method presented in \cite{BO10} is able to construct an
approximation of the matrix product in ${\mathcal O}(n k^2 \log n)$
operations, but requires ${\mathcal O}(n k^2 \log^2 n)$ operations
for the inverse or the LR factorization, i.e., it is not faster than
comparable algorithms for simple hierarchical matrices.

In very special cases, e.g., for essentially one-di\-men\-sio\-nal problems,
HSS-matrices \cite{CHGULY05,MAROTY05,CHGULIXI09b}, a special case of
${\mathcal H}^2$-matrices using a very simple block partition, allow
us to construct bases on the fly in optimal complexity.
Applications of HSS-matrices to two-dimensional geometries rely on
special properties to reduce to one-dimensional subproblems
\cite{MA08,XICHGULI09}, while no HSS-algo\-rithms of quasi-linear complexity
for three-dimensional problems are currently known.

In this paper, we present a new algorithm for approximating the
product of two ${\mathcal H}^2$-matrices.
Although this algorithm also requires ${\mathcal O}(n k^2 \log n)$
operations, it can be used to compute LR and Cholesky factorizations as
well as an approximate inverse in ${\mathcal O}(n k^2 \log n)$ operations,
making it significantly more efficient than previous methods.

The key feature of the new algorithm is the use of efficient
local low-rank updates:
given an ${\mathcal H}^2$-matrix $G$, we can add a low-rank
matrix $R\in\bbbr^{\hat t\times\hat s}$ to a submatrix
$G|_{\hat t\times\hat s}$ and ensure that the result is again
\emph{globally} an ${\mathcal H}^2$-matrix.
This update requires only ${\mathcal O}(k^2 (\#\hat t+\#\hat s))$
operations.
Since these update operations are at the heart of most important
algorithms for hierarchical matrices, the new algorithm allows
us to extend them to ${\mathcal H}^2$-matrices, reducing both
the algorithmic complexity and the storage requirements.

The paper is organized as follows:
the next section introduces the basic definitions of hierarchical
and ${\mathcal H}^2$-matrices.
It is followed by a section outlining the ideas of formatted
algebraic matrix operations that demonstrates that the most important
operations can be reduced to a sequence of local low-rank updates.
The next section describes how a low-rank update can be applied
efficiently to an ${\mathcal H}^2$-matrix, and it is followed by
a section considering the modifications required to perform
low-rank updates of submatrices efficiently.
The last section is devoted to numerical experiments illustrating
that the new algorithm takes ${\mathcal O}(n \log n)$ operations
to set up efficient preconditioners in ${\mathcal O}(n)$ units
of storage for FEM and BEM applications, i.e., that choosing an on
average constant rank is sufficient.

\section[H2-matrices]{${\mathcal H}^2$-matrices}

Hierarchical matrices, like most representation schemes for
rank-structured matrices, are based on a hierarchical subdivision
of the index set.
This subdivision can be expressed by a tree structure.

%
%
\begin{definition}[Cluster tree]
Let $\Idx$ be a finite index set, and let ${\mathcal T}$ be a
labeled tree.
We write $t\in\mathcal{T}$ if $t$ is a node in $\mathcal{T}$.
The set of sons of $t\in\mathcal{T}$ is denoted by
$\sons(t)\subseteq\mathcal{T}$, and the label of each node
$t\in{\mathcal T}$ by $\hat t$.

${\mathcal T}$ is called a \emph{cluster tree} for $\Idx$ if
the following conditions hold:
\begin{itemize}
  \item its root $r=\treeroot({\mathcal T})$ is labeled by $\Idx$,
        i.e., $\hat r=\Idx$,
  \item for each $t\in{\mathcal T}$ with $\sons(t)\neq\emptyset$,
        we have $\hat t = \bigcup_{t'\in\sons(t)} \hat t'$,
  \item for all $t\in{\mathcal T}$ and $t_1,t_2\in\sons(t)$
        with $t_1\neq t_2$, we have $\hat t_1\cap\hat t_2=\emptyset$.
\end{itemize}
A cluster tree for $\Idx$ is denoted by $\ctI$, its nodes are called
\emph{clusters}, and the set of its leaves is denoted by
$\lfI := \{ t\in\ctI\ :\ \sons(t)=\emptyset \}$.
\end{definition}

We note that the first two conditions imply $\hat t\subseteq\Idx$
for all $t\in\ctI$, and that all three conditions together imply that
the labels of the leaves form a disjoint partition $\{ \hat t\ :\ t\in\lfI \}$
of the index set $\Idx$.

Since we are interested in a partition of a matrix, we construct
a special cluster tree for product index sets $\Idx\times\Jdx$.

%
%
\begin{definition}[Block tree]
Let $\ctI$ and $\ctJ$ be cluster trees for index sets $\Idx$
and $\Jdx$, respectively.

A labeled tree ${\mathcal T}$ is called a \emph{block tree} for
$\ctI$ and $\ctJ$ if the following conditions hold:
\begin{itemize}
  \item for each $b\in{\mathcal T}$, there are $t\in\ctI$
        and $s\in\ctJ$ such that $b=(t,s)$ and $\hat b=\hat t\times\hat s$,
  \item the root $r=\treeroot({\mathcal T})$ of the tree $\mathcal{T}$
        is the pair of the roots of $\ctI$ and $\ctJ$, i.e.,
        $r=(\treeroot(\ctI),\treeroot(\ctJ))$,
  \item for each $b=(t,s)\in{\mathcal T}$ with $\sons(b)\neq\emptyset$,
        we have
        \begin{equation*}
          \sons(b) = \begin{cases}
            \{t\}\times\sons(s)
            &\text{ if } \sons(t)=\emptyset,\\
            \sons(t)\times\{s\}
            &\text{ if } \sons(s)=\emptyset,\\
            \sons(t)\times\sons(s)
            &\text{ otherwise}.
          \end{cases}
        \end{equation*}
\end{itemize}
A block tree for $\ctI$ and $\ctJ$ is denoted by $\ctIJ$, its nodes
are called \emph{blocks}, and the set of its leaves is denoted by
$\lfIJ := \{ b\in\ctIJ\ :\ \sons(b)=\emptyset \}$.

For each block $b=(t,s)\in\ctIJ$, we call $t$ the \emph{row cluster}
(or \emph{target cluster}) and $s$ the \emph{column cluster}
(or \emph{source cluster}).
\end{definition}

We note that the definitions imply that a block tree $\ctIJ$ is a
cluster tree for $\Idx\times\Jdx$ and that therefore the set
$\{ \hat t\times\hat s \ :\ b=(t,s)\in\lfIJ \}$ is a disjoint
partition of this index set.
We use this partition to split matrices into submatrices.

In order to determine which of these submatrices can be approximated
by low-rank representations, we split the set $\lfIJ$ of leaf blocks
into a set $\lfaIJ$ of \emph{admissible} blocks and the remainder
$\lfiIJ := \lfIJ\setminus\lfaIJ$ of \emph{inadmissible} blocks.

We represent the admissible blocks $b=(t,s)\in\lfaIJ$ in the form
\begin{equation*}
  G|_{\hat t\times\hat s} \approx V_t S_b W_s^*,
\end{equation*}
where the columns of $V_t$ and $W_s$ are interpreted as basis vectors
for subsets of $\bbbr^{\hat t}$ and $\bbbr^{\hat s}$ and $S_b$ contains
the coefficients corresponding to these basis vectors.
We require the basis vectors to be \emph{hierarchically nested}:

%
%
\begin{definition}[Cluster basis]
Let $k\in\bbbn$, and let $(V_t)_{t\in\ctI}$ be a family of matrices
satisfying $V_t\in\bbbr^{\hat t\times k}$ for all $t\in\ctI$.

This family is called a \emph{(nested) cluster basis} if for
each $t\in\ctI$ there is a matrix $E_t\in\bbbr^{k\times k}$ such
that
\begin{align}\label{eq:nested}
  V_t|_{\hat t'\times k} &= V_{t'} E_{t'} &
  &\text{ for all } t\in\ctI,\ t'\in\sons(t).
\end{align}
The matrices $E_t$ are called \emph{transfer matrices}, and
$k$ is called the \emph{rank} of the cluster basis.
\end{definition}

Due to (\ref{eq:nested}), we only have to store the matrices
$V_t$ for leaf clusters $t\in\lfI$ and the transfer matrices
$E_t$ for all clusters $t\in\ctI$.
In typical situations, this \emph{nested representation} requires
only ${\mathcal O}(n_\Idx k)$ units of storage \cite{BOHA02,BO10}, where
$n_\Idx:=\#\Idx$ denotes the cardinality of the index set $\Idx$.

%
%
\begin{definition}[${\mathcal H}^2$-matrix]
Let $\ctI$ and $\ctJ$ be cluster trees for index sets $\Idx$ and
$\Jdx$, let $\ctIJ$ be a block tree for $\ctI$ and $\ctJ$, and let
$(V_t)_{t\in\ctI}$ and $(W_s)_{s\in\ctJ}$ be cluster bases.

A matrix $G\in\bbbr^{\Idx\times\Jdx}$ is called an
\emph{${\mathcal H}^2$-matrix} for $\ctIJ$, $(V_t)_{t\in\ctI}$
and $(W_s)_{s\in\ctJ}$, if for each admissible block $b=(t,s)\in\lfaIJ$
there is a matrix $S_b\in\bbbr^{k\times k}$ such that
\begin{equation}\label{eq:vsw}
  G|_{\hat t\times\hat s} = V_t S_b W_s^*.
\end{equation}
The matrices $S_b$ are called \emph{coupling matrices}, the cluster
bases $(V_t)_{t\in\ctI}$ and $(W_s)_{s\in\ctJ}$ are called \emph{row and
column cluster bases}.
\end{definition}

We can represent an ${\mathcal H}^2$-matrix by the cluster bases,
the coupling matrices $S_b$ for all admissible blocks $b\in\lfaIJ$ and
the \emph{nearfield matrices} $G|_{\hat t\times\hat s}$ for all inadmissible
blocks $b=(t,s)\in\lfiIJ$.
In typical situations, this \emph{${\mathcal H}^2$-matrix representation}
requires only ${\mathcal O}((n_\Idx+n_\Jdx) k)$ units of storage
\cite{BOHA02,BO10}.

If we want to approximate a given matrix $G\in\bbbr^{\Idx\times\Jdx}$
by an ${\mathcal H}^2$-matrix $\widetilde G$ for cluster bases
$(V_t)_{t\in\ctI}$ and $(W_s)_{s\in\ctJ}$, \emph{orthogonal cluster bases}
are very useful:
if we assume
\begin{align*}
  V_t^* V_t &= I, &
  W_s^* W_s &= I &
  &\text{ for all } t\in\ctI,\ s\in\ctJ,
\end{align*}
the optimal coupling matrices (with respect both to the Frobenius norm
and the spectral norm) are given by
\begin{align*}
  S_b &:= V_t^* G|_{\hat t\times\hat s} W_s &
  &\text{ for all } b=(t,s)\in\lfaIJ.
\end{align*}
This property can be used to compute the best approximation of the
product of ${\mathcal H}^2$-matrices in ${\mathcal O}(n k^2)$ operations
\cite{BO04a} as long as both cluster bases are known in advance.

\section{Algebraic operations}
\label{se:algebraic_operations}

We are looking for a preconditioner for a matrix $A\in\bbbr^{\Idx\times\Idx}$
corresponding to a Galerkin discretization of an elliptic partial
differential equation or an integral equation.
According to \cite{LI04,BE05,GRKRLE05a,FAMEPR13,FAMEPR13b}, we can expect
to be able to approximate the matrices $L,R\in\bbbr^{\Idx\times\Idx}$ of the
standard triangular LR factorization by ${\mathcal H}^2$-matrices.

For the sake of brevity, we restrict our discussion to the case of
binary cluster trees, i.e., each cluster has either two sons or none.

Since $A$ is a quadratic matrix, we can use the same cluster tree $\ctI$
for its rows and columns.
We denote the corresponding block tree by $\ctII$.

The construction of the LR factorization is performed by a recursive
procedure:
let $t\in\ctI$.
We are looking for the LR factorization of $A|_{\hat t\times\hat t}$.

If $\sons(t)=\emptyset$, we can assume that $A|_{\hat t\times\hat t}$
is a small matrix, therefore we can construct the LR factorization
by standard Gaussian elimination.

If $\sons(t)=\{t_1,t_2\}$, on the other hand, we split the matrices
into submatrices
\begin{subequations}\label{eq:submatrices}
\begin{align}
  A_{11} &:= A|_{\hat t_1\times\hat t_1}, &
  A_{12} &:= A|_{\hat t_1\times\hat t_2},\\
  A_{21} &:= A|_{\hat t_2\times\hat t_1}, &
  A_{22} &:= A|_{\hat t_2\times\hat t_2},\\
  L_{11} &:= L|_{\hat t_1\times\hat t_1}, &
  L_{21} &:= L|_{\hat t_2\times\hat t_1},\\
  L_{22} &:= L|_{\hat t_2\times\hat t_2}, &
  R_{11} &:= R|_{\hat t_1\times\hat t_1},\\
  R_{12} &:= R|_{\hat t_1\times\hat t_2}, &
  R_{22} &:= R|_{\hat t_2\times\hat t_2}
\end{align}
\end{subequations}
and consider the block equation
\begin{align*}
  \begin{pmatrix}
    A_{11} & A_{12}\\
    A_{21} & A_{22}
  \end{pmatrix}
  &= A|_{\hat t\times\hat t}
   = L|_{\hat t\times\hat t} R|_{\hat t\times\hat t}\\
  &= \begin{pmatrix}
     L_{11} & \\
     L_{21} & L_{22}
   \end{pmatrix}
   \begin{pmatrix}
     R_{11} & R_{12}\\
     & R_{22}
   \end{pmatrix}
   = \begin{pmatrix}
     L_{11} R_{11} & L_{11} R_{12}\\
     L_{21} R_{11} & L_{22} R_{22} + L_{21} R_{12}
   \end{pmatrix}.
\end{align*}
We can solve $L_{11} R_{11} = A_{11}$ by recursion and obtain $L_{11}$
and $R_{11}$.
Then we can solve the triangular systems $L_{11} R_{12} = A_{12}$ and
$L_{21} R_{11} = A_{21}$ by forward substitution to obtain $R_{12}$
and $L_{21}$.
In a last step, we can solve $L_{22} R_{22} = A_{22} - L_{21} R_{12}$
by recursion to obtain $L_{22}$ and $R_{22}$, completing the algorithm.

The block forward substitution can also be handled by recursion:
in order to solve $L X = Y$ to obtain $X\in\bbbr^{\hat t\times\hat s}$
given $Y\in\bbbr^{\hat t\times\hat s}$, we let
\begin{align*}
  X_1 &:= X|_{\hat t_1\times\hat s}, &
  X_2 &:= X|_{\hat t_2\times\hat s},\\
  Y_1 &:= Y|_{\hat t_1\times\hat s}, &
  Y_2 &:= Y|_{\hat t_2\times\hat s}
\end{align*}
and consider
\begin{align*}
  \begin{pmatrix}
    Y_1\\ Y_2
  \end{pmatrix}
  &= Y = L|_{\hat t\times\hat t} X
   = \begin{pmatrix}
     L_{11} & \\
     L_{21} & L_{22}
   \end{pmatrix}
   \begin{pmatrix}
     X_1\\ X_2
   \end{pmatrix}
   = \begin{pmatrix}
     L_{11} X_1\\
     L_{21} X_1 + L_{22} X_2
   \end{pmatrix}.
\end{align*}
We can solve $L_{11} X_1 = Y_1$ by recursion and obtain $X_1$.
In the second step, we solve $L_{22} X_2 = Y_2 - L_{21} X_1$ again
by recursion and obtain $X_2$, completing the algorithm.

For the system $X R = Y$ with $X,Y\in\bbbr^{\hat s\times\hat t}$, we
can follow a similar approach:
we let
\begin{align*}
  X_1 &:= X|_{\hat s\times\hat t_1}, &
  X_2 &:= X|_{\hat s\times\hat t_2},\\
  Y_1 &:= Y|_{\hat s\times\hat t_1}, &
  Y_2 &:= Y|_{\hat s\times\hat t_2}
\end{align*}
and arrive at the block equation
\begin{align*}
  \begin{pmatrix}
    Y_1 & Y_2
  \end{pmatrix}
  &= Y = X R|_{\hat t\times\hat t}
   = \begin{pmatrix}
      X_1 & X_2
    \end{pmatrix}
    \begin{pmatrix}
      R_{11} & R_{12}\\
      & R_{22}
    \end{pmatrix}
   = \begin{pmatrix}
      X_1 R_{11} & X_1 R_{12} + X_2 R_{22}
   \end{pmatrix},
\end{align*}
requiring us to solve $X_1 R_{11} = Y_1$ and $X_2 R_{22} = Y_2 - X_1 R_{12}$
by recursion.

We conclude that we only need an efficient algorithm for computing
the matrix product in order to derive efficient algorithms for the
block forward substitution and the LR factorization.
The matrix inverse can be constructed by a similar procedure
(cf. Remark~\ref{rm:inversion}).

Both the inversion and the LR factorization require updates to
submatrices of the form
\begin{equation*}
  Z|_{\hat t\times\hat r}
  \gets Z|_{\hat t\times\hat r}
          + \alpha X|_{\hat t\times\hat s} Y|_{\hat s\times\hat r}
\end{equation*}
with $(t,r),(t,s),(s,r)\in\ctIJ$ and $\alpha\in\bbbr$.
If both $(t,s)$ and $(s,r)$ are not leaves, we can use recursion
once again.
For the sake of brevity, we consider only the most general case that
none of the clusters is a leaf, i.e., that we have
$\sons(t)=\{t_1,t_2\}$, $\sons(s)=\{s_1,s_2\}$ and $\sons(r)=\{r_1,r_2\}$.
As before, we split $X|_{\hat t\times\hat s}$, $Y|_{\hat s\times\hat r}$ and
$Z|_{\hat t\times\hat r}$ into submatrices
\begin{gather*}
  X|_{\hat t\times\hat s} = \begin{pmatrix}
    X_{11} & X_{12}\\
    X_{21} & X_{22}
  \end{pmatrix},\qquad
  Y|_{\hat s\times\hat r} = \begin{pmatrix}
    Y_{11} & Y_{12}\\
    Y_{21} & Y_{22}
  \end{pmatrix},\qquad
  Z|_{\hat t\times\hat r} = \begin{pmatrix}
    Z_{11} & Z_{12}\\
    Z_{21} & Z_{22}
  \end{pmatrix}
\end{gather*}
and see that the product can be computed by recursively
carrying out the updates
\begin{align*}
  Z_{11} &\gets Z_{11} + \alpha X_{11} Y_{11}, &
  Z_{12} &\gets Z_{12} + \alpha X_{11} Y_{12},\\
  Z_{21} &\gets Z_{21} + \alpha X_{21} Y_{11}, &
  Z_{22} &\gets Z_{22} + \alpha X_{21} Y_{12},\\[3pt]
  Z_{11} &\gets Z_{11} + \alpha X_{12} Y_{21}, &
  Z_{12} &\gets Z_{12} + \alpha X_{12} Y_{22},\\
  Z_{21} &\gets Z_{21} + \alpha X_{22} Y_{21}, &
  Z_{22} &\gets Z_{22} + \alpha X_{22} Y_{22}.
\end{align*}
If $(t,r)$ is a leaf, we can temporarily create submatrices
$Z_{11},Z_{12},Z_{21}$ and $Z_{22}$ to receive the result of the
recursive procedure.
These submatrices can then be extended with zeros and added to the
block $Z|_{\hat t\times\hat r}$.

If $(t,s)$ or $(s,r)$ is a leaf block, the term $X|_{\hat t\times\hat s}
Y|_{\hat s\times\hat r}$ is a low-rank matrix, and a factorized
representation can be obtained easily by multiplying the factorized
representation of the low-rank matrix with the other matrix
\cite[Sections 3.2 and 3.3]{BOHA02}.

In both cases, we need an efficient algorithm that adds a low-rank
matrix in factorized representation to an existing ${\mathcal H}^2$-matrix.
For hierarchical matrices, this task can be handled by a relatively
simple approach:
we split the low-rank matrix into submatrices matching the block
structure of the target matrix, then perform a truncated addition
of the submatrices using the singular value decomposition.

Adding a low-rank matrix to a given $\mathcal{H}^2$-matrix is considerably
more challenging, since the blocks share row and column cluster bases,
so changing one block changes the entire corresponding block row and column.

\section{Recompression}

We first consider a \emph{global} low-rank update.
Let $Z\in\bbbr^{\Idx\times\Jdx}$ be an ${\mathcal H}^2$-matrix with
row cluster basis $(V_t)_{t\in\ctI}$ and column cluster basis
$(W_s)_{s\in\ctJ}$ for a block tree $\ctIJ$.

Let $X\in\bbbr^{\Idx\times k}$ and $Y\in\bbbr^{\Jdx\times k}$.
We are looking for an efficient algorithm for approximating the
sum $\widetilde Z := Z + X Y^*$.

Our starting point is a simple observation:
for each admissible block $b=(t,s)\in\lfaIJ$, we have
\begin{align*}
  \widetilde{Z}|_{\hat t\times\hat s} &= (Z + X Y^*)|_{\hat t\times\hat s}
   = Z|_{\hat t\times\hat s} + X|_{\hat t\times k} Y|_{\hat s\times k}^*\\
  &= V_t S_b W_s^* + X|_{\hat t\times k} I Y|_{\hat s\times k}^*\\
  &= \begin{pmatrix}
       V_t & X|_{\hat t\times k}
     \end{pmatrix}
     \begin{pmatrix}
       S_b & \\
       & I
     \end{pmatrix}
     \begin{pmatrix}
       W_s & Y|_{\hat s\times k}
     \end{pmatrix}^*,
\end{align*}
i.e., all submatrices are already given in factorized form.
By introducing new cluster bases and coupling matrices
\begin{subequations}\label{eq:tildeZ_h2}
\begin{align}
  \widetilde V_t &:= \begin{pmatrix}
    V_t & X|_{\hat t\times k}
  \end{pmatrix}, &
  &\text{ for all } t\in\ctI,\\
  \widetilde W_s &:= \begin{pmatrix}
    W_s & Y|_{\hat s\times k}
  \end{pmatrix}, &
  &\text{ for all } s\in\ctJ,\\
  \widetilde S_b &:= \begin{pmatrix}
    S_b & \\
    & I
  \end{pmatrix} &
  &\text{ for all } b\in\lfaIJ,
\end{align}
\end{subequations}
we have found an \emph{exact} ${\mathcal H}^2$-matrix representation
of $\widetilde Z = Z+X Y^*$.
The transfer matrices for $(\widetilde V_t)_{t\in\ctI}$ and
$(\widetilde W_s)_{s\in\ctJ}$ can be easily obtained by extending
the transfer matrices of the original cluster bases with the
identity matrix, similar to the way the coupling matrices $S_b$
are extended to $\widetilde S_b$.

Unfortunately the rank of this representation $\widetilde Z$ has doubled
compared to the original ${\mathcal H}^2$-matrix $Z$, therefore using it
repeatedly as required by the matrix multiplication would quickly lead
to very large ranks and a very inefficient algorithm.

We can fix this problem by using a \emph{recompression algorithm}
that takes an ${\mathcal H}^2$-matrix and approximates it by an
${\mathcal H}^2$-matrix of lower rank.
This closely resembles the truncation operation that is at the heart of
algebraic operations for hierarchical matrices.

An appropriate algorithm is described in \cite[Section~6.6]{BO10}:
we aim to construct an improved orthogonal row cluster basis
$(Q_t)_{t\in\ctI}$.
A column cluster basis can be handled similarly, replacing
$\widetilde Z$ by $\widetilde Z^*$.
The matrix $Q_t$ should be chosen in such a way that all blocks
connected to the cluster $t$ or one of its predecessors
\begin{equation*}
  \pred(t) := \begin{cases}
    \{ t \} &\text{ if } t=\treeroot(\ctI),\\
    \{ t \} \cup \pred(t^+) &\text{ for } t^+ \text{ with } t\in\sons(t^+)
  \end{cases}
\end{equation*}
can be approximated in the range of $Q_t$.
If we denote the set of all these clusters by
\begin{align*}
  \row^*(t) &:= \bigcup_{t^*\in\pred(t)} \row(t^*),\\
  \row(t) &:= \{ s\in\ctJ\ :\ (t,s)\in\lfaIJ \},
\end{align*}
and take advantage of the fact that we are looking for an orthogonal
cluster basis, we have to ensure
\begin{align*}
  Q_t Q_t^* \widetilde Z|_{\hat t\times\hat s}
  &\approx \widetilde Z|_{\hat t\times\hat s} &
  &\text{ for all } s\in\row^*(t).
\end{align*}
The task of finding an orthogonal basis that approximates multiple
matrices simultaneously can be simplified by combining all of these
matrices in a large matrix:
we let $\tau:=\#\row^*(t)$ and $\row^*(t)=\{s_1,\ldots,s_\tau\}$ and
introduce the matrix
\begin{equation*}
  \widetilde Z_t := \begin{pmatrix}
    \widetilde Z|_{\hat t\times\hat s_1} & \ldots &
    \widetilde Z|_{\hat t\times\hat s_\tau}
  \end{pmatrix}.
\end{equation*}
If we have found an orthogonal matrix $Q_t$ of low rank satisfying
\begin{equation}\label{eq:Qt_condition}
  Q_t Q_t^* \widetilde Z_t \approx \widetilde Z_t,
\end{equation}
we have solved our problem.
Applying the singular value decomposition directly to this task would
lead to an algorithm of at least quadratic complexity.

Fortunately, we can take advantage of the fact that $\widetilde Z_t$
is an ${\mathcal H}^2$-matrix:
for each $s\in\row(t)$, we have $(t,s)\in\lfaIJ$ by definition
and therefore
\begin{equation*}
  \widetilde Z|_{\hat t\times\hat s}
  = \widetilde V_t \widetilde S_b \widetilde W_s^*
  = \widetilde V_t B_{t,s}^*
\end{equation*}
with $B_{t,s} := \widetilde W_s \widetilde S_b^* \in \bbbr^{\hat s\times (2k)}$.
Since the cluster basis $(\widetilde V_t)_{t\in\ctI}$ is nested, we can
extend this result:
for each $s\in\row^*(t)$, we find a matrix $B_{t,s}\in\bbbr^{\hat s\times(2k)}$
with
\begin{equation*}
  \widetilde Z|_{\hat t\times\hat s}
  = \widetilde V_t B_{t,s}^*,
\end{equation*}
therefore we have
\begin{align}
  \widetilde Z_t
  &= \begin{pmatrix}
       \widetilde Z|_{\hat t\times\hat s_1} & \ldots &
       \widetilde Z|_{\hat t\times\hat s_\tau}
     \end{pmatrix}
   = \begin{pmatrix}
       \widetilde V_t B_{t,s_1}^* & \ldots & \widetilde V_t B_{t,s_\tau}^*
     \end{pmatrix}
   = \widetilde V_t B_t^*,\label{eq:Zt_factors}\\
\intertext{with}
  B_t &:= \begin{pmatrix}
    B_{t,s_1}\\ \vdots\\ B_{t,s_\tau}
  \end{pmatrix}.\notag
\end{align}
We conclude that the matrices $\widetilde Z_t$ are of low rank and
given in factorized form.

If we solve (\ref{eq:Qt_condition}) by computing the singular value
decomposition of $\widetilde Z_t$, we only require the left singular
vectors and the singular values to construct $Q_t$.
This means that applying orthogonal transformations to the columns
of $\widetilde Z_t$ will not change the result of our algorithm.
We can take advantage of this property to reduce the computational
work:
let $P_t \widetilde B_t = B_t$ be a QR decomposition of $B_t$ with
$\widetilde B_t\in\bbbr^{(2k)\times(2k)}$ and an orthogonal matrix $P_t$.
Replacing $B_t^*$ by $\widetilde B_t^* P_t^*$ in (\ref{eq:Zt_factors})
yields that (\ref{eq:Qt_condition}) is equivalent to
\begin{equation}\label{eq:Qt_condensed}
  Q_t Q_t^* \widetilde V_t \widetilde B_t^*
  \approx \widetilde V_t \widetilde B_t^*,
\end{equation}
so $Q_t$ can be obtained by computing the singular value decomposition
of the matrix $\widetilde V_t \widetilde B_t^*$ with only $2k$ columns.
The \emph{weight matrices} $\widetilde B_t$ capture the re\-la\-tive
importance of the columns of $\widetilde V_t$ and are very important
for controlling the approximation error.

Finding the QR decomposition $P_t \widetilde B_t = B_t$ directly
would again lead to an algorithm of quadratic complexity, but
we can once more take advantage of the properties of
${\mathcal H}^2$-matrices:
the cluster basis $(\widetilde V_t)_{t\in\ctI}$ is nested.
If we have $s\in\row^*(t)$ with $(t,s)\not\in\lfaIJ$, the definition
implies $s\in\row^*(t^+)$, where $t^+$ denotes the father of $t$.
Therefore we can find $B_{t^+,s}\in\bbbr^{\hat s\times(2k)}$ such that
\begin{equation*}
  \widetilde Z|_{\hat t^+\times\hat s}
  = \widetilde V_{t^+} B_{t^+,s}^*,
\end{equation*}
and (\ref{eq:nested}) yields
\begin{equation*}
  \widetilde Z|_{\hat t\times\hat s}
  = \widetilde V_{t^+}|_{\hat t\times k} B_{t^+,s}^*
  = \widetilde V_t \widetilde E_t B_{t^+,s}^*,
\end{equation*}
where $\widetilde E_t$ denotes the transfer matrix corresponding
to the cluster $t$.

This means that all blocks connected to ``strict'' predecessors
of $t$ are also present in $B_{t^+}$, and at least part of the
QR decomposition of $B_t$ can be inherited from the father $t^+$.
Let $\varrho := \#\row(t)$ and $\{s_1,\ldots,s_\varrho\}:=\row(t)$.
Then we can write $B_t$ as
\begin{equation*}
  B_t
  = \begin{pmatrix}
      B_{t,s_1}\\
      \vdots\\
      B_{t,s_\varrho}\\
      B_{t^+} \widetilde E_t^*
    \end{pmatrix}
  = \begin{pmatrix}
      B_{t,s_1}\\
      \vdots\\
      B_{t,s_\varrho}\\
      P_{t^+} \widetilde{B}_{t^+} E_t^*
    \end{pmatrix}.
\end{equation*}
Due to $(t,s_i)\in\lfaIJ$, we have
$B_{t,s_i} = \widetilde{W}_{s_i} \widetilde{S}_{t,s_i}^*$
for all $i\in\{1,\ldots,\varrho\}$.
Using a simple recursion \cite[Algorithm~16]{BO10}, we can obtain
thin QR factorizations $\widetilde{W}_s = P_{W,s} R_{W,s}$ with
$R_{W,s}\in\bbbr^{(2k)\times(2k)}$ and write $B_t$ as
\begin{align*}
  B_t
  &= \begin{pmatrix}
       P_{W,s_1} R_{W,s_1} S_{t,s_1}^*\\
       \vdots\\
       P_{W,s_\varrho} R_{W,s_\varrho} S_{t,s_\varrho}^*\\
       P_{t^+} \widetilde{B}_{t^+} E_t^*
     \end{pmatrix}
   = \begin{pmatrix}
       P_{W,s_1} & & & \\
       & \ddots & & \\
       & & P_{W,s_\varrho} & \\
       & & & P_{t^+}
     \end{pmatrix}
     \begin{pmatrix}
       R_{W,s_1} S_{t,s_1}^*\\
       \vdots\\
       R_{W,s_\varrho} S_{t,s_\varrho}^*\\
       \widetilde{B}_{t^+} E_t^*
     \end{pmatrix}.
\end{align*}
The left factor is already orthogonal, therefore finding the QR
decomposition of $B_t$ only requires us to find the decomposition
of the right factor with $2k$ columns and $2k(\varrho+1)$ rows.
Since our algorithm does not require the matrices $P_t$ or $P_{W,s}$,
they are not computed.
This leads to a recursive construction \cite[Algorithm~28]{BO10}
that computes all weight matrices $(\widetilde{B}_t)_{t\in\ctI}$
in $\mathcal{O}(n_\Idx k^2)$ operations.

For an ${\mathcal H}^2$-matrix representation, we require a
\emph{nes\-ted} cluster basis, and we briefly summarize an algorithm
\cite[Algorithm~30]{BO10} that directly computes the transfer matrices:
if $t$ is a leaf, we compute the singular value decomposition
of $\widetilde V_t \widetilde B_t^*$ directly and use the
left singular vectors corresponding to the largest singular values
to construct $Q_t$.

If $t$ is not a leaf, we have $\sons(t)=\{t_1,t_2\}$ and compute
$Q_{t_1}$ and $Q_{t_2}$ by recursion.
Since we are looking for a nested cluster basis, we only have to
construct transfer matrices $F_{t_1}$ and $F_{t_2}$ with
\begin{align}\label{eq:Qt_nested}
  Q_t &= \begin{pmatrix}
    Q_{t_1} F_{t_1}\\
    Q_{t_2} F_{t_2}
  \end{pmatrix}
  = \begin{pmatrix}
    Q_{t_1} & \\
    & Q_{t_2}
  \end{pmatrix} \widehat Q_t, &
  \widehat Q_t &:= \begin{pmatrix}
    F_{t_1}\\ F_{t_2}
  \end{pmatrix}.
\end{align}
This leads to
\begin{align*}
  \begin{pmatrix}
    Q_{t_1} & \\
    & Q_{t_2}
  \end{pmatrix} \widehat Q_t
  \widehat Q_t^*
  \begin{pmatrix}
    Q_{t_1}^* & \\
    & Q_{t_2}^*
  \end{pmatrix} \widetilde V_t \widetilde B_t^*
  &= Q_t Q_t^* \widetilde V_t \widetilde B_t^*
   \approx \widetilde V_t \widetilde B_t^*,
\end{align*}
and multiplying both sides with the adjoints $Q_{t_1}^*$ and
$Q_{t_2}^*$ yields
\begin{align*}
  \widehat Q_t \widehat Q_t^*
  \begin{pmatrix}
    Q_{t_1}^* \widetilde V_t|_{\hat t_1\times k}\\
    Q_{t_2}^* \widetilde V_t|_{\hat t_2\times k}
  \end{pmatrix}
  \widetilde B_t^*
  &\approx \begin{pmatrix}
             Q_{t_1}^* & \\
             & Q_{t_2}^*
           \end{pmatrix}
   \widetilde V_t \widetilde B_t^*
   = \begin{pmatrix}
       Q_{t_1}^* \widetilde V_t|_{\hat t_1\times k}\\
       Q_{t_2}^* \widetilde V_t|_{\hat t_2\times k}
     \end{pmatrix} \widetilde B_t^*.
\end{align*}
Since $(\widetilde V_t)_{t\in\ctI}$ is a \emph{nested} cluster basis,
we can use (\ref{eq:nested}) to obtain
\begin{equation*}
  \widehat V_t
  := \begin{pmatrix}
    Q_{t_1}^* \widetilde V_t|_{\hat t_1\times k}\\
    Q_{t_2}^* \widetilde V_t|_{\hat t_2\times k}
  \end{pmatrix}
  = \begin{pmatrix}
    Q_{t_1}^* \widetilde V_{t_1} \widetilde E_{t_1}\\
    Q_{t_2}^* \widetilde V_{t_2} \widetilde E_{t_2}
  \end{pmatrix}
\end{equation*}
and arrive at the approximation
\begin{equation}\label{eq:hatQt_condensed}
  \widehat Q_t \widehat Q_t^* \widehat V_t \widetilde B_t^*
  \approx \widehat V_t \widetilde B_t^*.
\end{equation}
It is of the same form as (\ref{eq:Qt_condensed}), and we can
again use the singular value decomposition to construct the
matrix $\widehat Q_t$.
Splitting $\widehat Q_t$ according to (\ref{eq:Qt_nested}) yields
the required transfer matrices.
Since $\widehat V_t \widetilde B_t^*$ is only a $(2k)\times(2k)$
matrix, the recursive algorithm requires only ${\mathcal O}(n_\Idx k^2)$
operations \cite[Theorem~6.27]{BO10} to construct the entire cluster
basis $(Q_t)_{t\in\ctI}$.

In order to construct the matrices $\widehat V_t$ in
(\ref{eq:hatQt_condensed}) efficiently, the algorithm computes
and stores the matrices $R_t := Q_t^* \widetilde V_t$ that can
be obtained in ${\mathcal O}(k^3)$ operations via
$R_t = \widehat Q_t^* \widehat V_t$ if $t$ is not a leaf.
These matrices are also useful for converting the matrix $\widetilde Z$
to the new basis:
since $Q_t$ is an orthogonal matrix, the best approximation of
$\widetilde{V}_t \widetilde{S}_b \widetilde{W}_s^*$ in the new basis
is given by the orthogonal projection
$Q_t Q_t^* \widetilde{V}_t \widetilde{S}_b \widetilde{W}_s^*
= Q_t R_t \widetilde{S}_b \widetilde{W}_s^*$, so a multiplication
of $\widetilde{S}_b$ with the small matrix $R_t$ is sufficient to
switch a block to the new basis.
Applying this procedure to the entire matrix takes
${\mathcal O}(n_\Idx k^2)$ operations.
The corresponding algorithm for the column basis requires
$\mathcal{O}(n_\Jdx k^2)$ operations.

There are various strategies for choosing the truncation accuracies
in (\ref{eq:Qt_condensed}) and (\ref{eq:hatQt_condensed}), ranging from
simply using a constant accuracy for all clusters to more subtle techniques
that include weighting factors in the matrices $B_{t,s}$ to obtain the
equivalent of variable-order schemes \cite{BOLOME02,BO05a} or to ensure
blockwise relative error bounds \cite[Section~6.8]{BO10}.

To summarize:
the matrix $\widetilde Z = Z + X Y^*$ can be expressed as an
${\mathcal H}^2$-matrix of rank $2k$ using (\ref{eq:tildeZ_h2}),
and the recompression algorithm can be used to construct an approximation
of reduced rank in ${\mathcal O}((n_\Idx+n_\Jdx) k^2)$ operations.
This means that \emph{global} low-rank updates can be carried
out in linear complexity.

\section{Local updates}

The multiplication algorithm applies low-rank updates to
\emph{submatrices}, not to the entire matrix, therefore we have
to modify the algorithm outlined in the previous section.

Let us assume that we want to add a low-rank matrix $X Y^*$ with
$X\in\bbbr^{\hat t_0\times k}$ and $Y\in\bbbr^{\hat s_0\times k}$ to a matrix
block $Z|_{\hat t_0\times\hat s_0}$ with $b_0=(t_0,s_0)\in\ctIJ$ and leave
the remainder of the matrix $Z$ essentially unchanged.
A simple algorithm would be to extend $X$ and $Y$ by zero and
use the algorithm presented before, but this would lead to a
relatively high computational complexity.

If $Z$ is an ${\mathcal H}^2$-matrix, the submatrix
$Z|_{\hat t_0\times\hat s_0}$ is also an ${\mathcal H}^2$-matrix with
cluster bases and coupling matrices taken from subtrees:
let $\ctt$ denote the subtree of $\ctI$ with root $t_0$,
$\cts$ the subtree of $\ctJ$ with root $s_0$, and $\ctts$ the subtree
of $\ctIJ$ with root $b_0=(t_0,s_0)$.
$Z|_{\hat t_0\times\hat s_0}$ is an ${\mathcal H}^2$-matrix
for the block tree $\ctts$ with row basis $(V_t)_{t\in\ctt}$ and
column basis $(W_s)_{s\in\cts}$.
Therefore we can apply the algorithm given above and update the block
in ${\mathcal O}((\#\hat t_0 + \#\hat s_0) k^2)$ operations.

Unfortunately, the result will in general no longer be an
${\mathcal H}^2$-matrix, since $Z|_{(\Idx\setminus\hat t_0)\times\hat s_0}$
and $Z|_{\hat t_0\times(\Jdx\setminus\hat s_0)}$ would still be represented by
the original cluster bases, not by the ones constructed for the
update.

We can fix this problem by switching \emph{all} blocks to the new
cluster bases:
Applying the update to $Z|_{\hat t_0\times\hat s_0}$ yields new cluster
bases $(Q_t)_{t\in\ctt}$ and $(Q_s)_{s\in\cts}$ and, as mentioned above,
also matrices $R_t=Q_t^* \widetilde V_t$ and $R_s=Q_s^* \widetilde W_s$
describing the change of basis.
Using these matrices, we can update the coupling matrices:
\begin{align*}
  S_b &\gets R_t \begin{pmatrix} S_b & 0\\ 0 & I \end{pmatrix} R_s^* &
  &\text{ if } t\in\ctt,\ s\in\cts,\\
  S_b &\gets R_t \begin{pmatrix} S_b\\ 0 \end{pmatrix} &
  &\text{ if } t\in\ctt,\ s\not\in\cts,\\
  S_b &\gets \begin{pmatrix} S_b & 0 \end{pmatrix} R_s^* &
  &\text{ if } t\not\in\ctt,\ s\in\cts.
\end{align*}
In order to obtain an ${\mathcal H}^2$-matrix, we also have to
update the cluster bases $(V_t)_{t\in\ctI}$ and $(W_s)_{s\in\ctJ}$.
Thanks to the nested representation of these bases, this update
is particularly simple:
we replace the original cluster bases in the relevant subtrees
by the new bases
\begin{align*}
  V_t &\gets Q_t &
  &\text{ for all } t\in\ctt,\\
  W_s &\gets Q_s &
  &\text{ for all } s\in\cts.
\end{align*}
and update the transfer matrices by
\begin{align*}
  E_{t_0} &\gets R_{t_0} \begin{pmatrix} E_{t_0}\\ 0 \end{pmatrix}, &
  E_{s_0} &\gets R_{s_0} \begin{pmatrix} E_{s_0}\\ 0 \end{pmatrix}
\end{align*}
in order to ensure that the resulting cluster bases are still
properly nested.

We do not have to change coupling matrices connected to proper
predecessors of $t_0$ or $s_0$, since the corresponding cluster bases
implicitly inherit the update via the transfer matrices $E_{t_0}$
and $E_{s_0}$.

In order to ensure that the new cluster bases are able to approximate
blocks outside of $\ctts$, we have to make sure that the weight
matrices $\widetilde B_t$ are computed correctly:
they have to take \emph{all} blocks $(t,s)$ with $s\in\row(t)$
into account, not only the blocks in $\ctts$.

Applying \cite[Lemma~6.26]{BO10} to the subtrees $\ctt$ and $\cts$
yields that the construction of the weight matrices takes only
${\mathcal O}((\#\hat t_0+\#\hat s_0) k^2)$ operations.
Updating the global cluster bases means only changing the two transfer
matrices $E_{t_0}$ and $E_{s_0}$, and this can be accomplished in
${\mathcal O}(k^3)$ operations.
Updating the coupling matrices requires not more than
\begin{align*}
  \sum_{\substack{(t,s)\in\lfaIJ\\ t\in\ctt}} \kern-10pt (2k)^3
       &+ \sum_{\substack{(t,s)\in\lfaIJ\\ s\in\cts}} \kern-10pt (2k)^3\\
  &\leq 8 C_{\rm sp} \sum_{t\in\ctt} k^3
       + 8 C_{\rm sp} \sum_{s\in\cts} k^3
   = 8 C_{\rm sp} (\#\ctt + \#\cts) k^3
\end{align*}
operations, where $C_{\rm sp}$ is the sparsity constant
\cite{BOHA02} of the block tree $\ctIJ$ satisfying
\begin{align*}
  \#\{ s\in\ctJ\ :\ (t,s)\in\ctIJ \} &\leq C_{\rm sp} &
  &\text{ for all } t\in\ctI,\\
  \#\{ t\in\ctI\ :\ (t,s)\in\ctIJ \} &\leq C_{\rm sp} &
  &\text{ for all } s\in\ctJ.
\end{align*}
With the standard assumptions $k^3 \#\ctt \in {\mathcal O}(k^2 \#\hat t_0)$
and $k^3 \#\cts \in {\mathcal O}(k^2 \#\hat s_0)$, we conclude that these
updates also take no more than ${\mathcal O}((\#\hat t_0+\#\hat s_0) k^2)$
operations.

In summary, we have proven the following complexity estimate:

%
%
\begin{theorem}[Complexity]
\label{th:lowrank_complexity}
Approximating the low-rank update $Z|_{\hat t_0\times\hat s_0} \gets
Z|_{\hat t_0\times\hat s_0} + X Y^*$ of a submatrix with
$b_0=(t_0,s_0)\in\ctIJ$  requires ${\mathcal O}((\#\hat t_0+\#\hat s_0) k^2)$
operations.
\end{theorem}

\section{Complexity of the LR factorization}

Having established an estimate for the complexity of low-rank
updates, we can now turn our attention towards complexity estimates
for the matrix multiplication, forward substitution and the
LR factorization.

Let $t,s,r\in\ctI$ such that $(t,s)\in\ctII$ and $(s,r)\in\ctII$.
We denote the number of operations required to perform the
update
\begin{equation*}
  Z|_{\hat t\times\hat r} \gets Z|_{\hat t\times\hat r}
        + X|_{\hat t\times\hat s} Y|_{\hat s\times\hat r}
\end{equation*}
by $W_{\rm mm}(t,s,r)$.

%
\paragraph{Case 1: $(t,s)$ and $(s,r)$ are subdivided.}

If $(t,s)$ and $(s,r)$ are not leaves of $\ctII$, we use recursion
to perform updates
\begin{equation*}
  Z|_{\hat t'\times\hat r'} \gets Z|_{\hat t'\times\hat r'}
        + X|_{\hat t'\times\hat s'} Y|_{\hat s'\times\hat r'}
\end{equation*}
for all $t'\in\sons^+(t)$, $s'\in\sons^+(s)$ and $r'\in\sons^+(r)$,
where the abbreviation
\begin{align*}
  \sons^+(t) &:= \begin{cases}
    \sons(t) &\text{ if } \sons(t)\neq\emptyset,\\
    \{t\} &\text{ otherwise}
  \end{cases} &
  &\text{ for all } t\in\ctI
\end{align*}
is convenient to express $\sons(t,s)=\sons^+(t)\times\sons^+(s)$.
If $(t,r)$ is also not a leaf of $\ctII$, no additional algebraic
operations are required and we obtain
\begin{equation}\label{eq:mm_recursion}
  W_{\rm mm}(t,s,r)
  = \sum_{\substack{t'\in\sons^+(t),\\s'\in\sons^+(s),\\r'\in\sons^+(r)}}
      W_{\rm mm}(t',s',r').
\end{equation}
If $(t,r)$ is a leaf, we create temporary submatrices for each
block $(t',r')$ with $t'\in\sons^+(t)$ and $r'\in\sons^+(r)$,
fill them as before using recursion, and then extend them by
zero and add them to $Z|_{\hat t\times\hat r}$.

According to Theorem~\ref{th:lowrank_complexity}, we can find
a constant $C_{\rm up}$ such that the low-rank update of the
block $\hat t\times\hat r$ requires not more than
\begin{equation*}
  C_{\rm up} k^2 (\#\hat t + \#\hat r)
\end{equation*}
operations and conclude
\begin{align*}
  W_{\rm mm}(t,s,r)
  &\leq 4 C_{\rm up} k^2 (\#\hat t + \#\hat r)
        + \sum_{\substack{t'\in\sons^+(t),\\s'\in\sons^+(s),\\r'\in\sons^+(r)}}
            W_{\rm mm}(t',s',r'),
\end{align*}
where the additional term is due to the fact that the four temporary
submatrices for $(t_1,s_1)$, $(t_1,s_2)$, $(t_2,s_1)$ and $(t_2,s_2)$
have to be added.

%
\paragraph{Case 2: $(t,s)$ or $(s,r)$ is a leaf.}

If $(s,r)$ is a leaf, we have
$Y|_{\hat s\times\hat r} = V_s S_{s,r} W_r^*$ and obtain
\begin{equation*}
  X|_{\hat t\times\hat s} Y|_{\hat s\times\hat r}
  = X|_{\hat t\times\hat s} V_s S_{s,r} W_r^*
  = \underbrace{X|_{\hat t\times\hat s} V_s}_{=:\widetilde X_{t,s}}
        S_{s,r} W_r^*.
\end{equation*}
Computing $\widetilde X_{t,s} = X|_{\hat t\times\hat s} V_s$ takes
$k$ matrix-vector multiplications with the ${\mathcal H}^2$-matrix
$X|_{\hat t\times\hat s}$, and each of these requires only
${\mathcal O}(k (\#\hat t+\#\hat s))$ operations
\cite[Theorem~3.42]{BO10}.
Multiplying $\widetilde X_{t,s}$ by $S_{s,r}$ requires
${\mathcal O}(k^2 \#\hat t)$ operations, and the low-rank update
\begin{equation*}
  Z|_{\hat t\times\hat r} \gets
    Z|_{\hat t\times\hat r} + (\widetilde X_{t,s} S_{s,r}) W_r^*
\end{equation*}
can be accomplished in not more than $C_{\rm up} k^2 (\#\hat t+\#\hat r)$
operations due to Theorem~\ref{th:lowrank_complexity}.
We conclude that there is a constant $C_{\rm lb}$ such that
\begin{equation*}
  W_{\rm mm}(t,s,r) \leq C_{\rm lb} k^2 (\#\hat t + \#\hat s + \#\hat r)
\end{equation*}
holds in this case.
By similar arguments we get the same estimate for the case that
$(t,s)$ is a leaf.

%
\paragraph{Matrix multiplication.}

Combining Case~1 and Case~2 yields a constant $C_{\rm mb}$ such that
\begin{align*}
  W_{\rm mm}(t,s,r)
  &\leq C_{\rm mb} k^2 (\#\hat t + \#\hat s + \#\hat r)
\intertext{if $(t,s)\in\lfII$ or $(s,r)\in\lfII$ and}
  W_{\rm mm}(t,s,r)
  &\leq C_{\rm mb} k^2 (\#\hat t + \#\hat s + \#\hat r)
         + \sum_{\substack{t'\in\sons^+(t),\\
                          s'\in\sons^+(s),\\
                          r'\in\sons^+(r)}}
           W_{\rm mm}(t',s',r')
\end{align*}
otherwise.
To obtain an estimate for the total complexity, we follow the approach
outlined in \cite[Section~7.7]{BO10}:
we collect the triples $(t,s,r)\in\ctI\times\ctI\times\ctI$ of clusters
for which multiplications are carried out in a \emph{triple tree}
$\ctIII$.
Its root is given by
$\treeroot(\ctIII)=(\treeroot(\ctI),\treeroot(\ctI),\treeroot(\ctI))$.
In Case~1, i.e., if $(t,s)\not\in\lfII$ and $(s,r)\not\in\lfII$, we let
\begin{equation*}
  \sons(t,s,r) := \sons^+(t)\times\sons^+(s)\times\sons^+(r).
\end{equation*}
Otherwise we have Case~2, i.e., $(t,s)\in\lfII$ or $(s,r)\in\lfII$,
and since no recursion takes place, we let
\begin{equation*}
  \sons(t,s,r) := \emptyset.
\end{equation*}
Using the triple tree, our complexity estimate can be expressed in
the form
\begin{align*}
  W_{\rm mm}(t_0,s_0,r_0)
  &\leq C_{\rm mb} k^2 \kern-20pt \sum_{(t,s,r)\in\cttsr} \kern-20pt
             (\#\hat t + \#\hat s + \#\hat r) &
  &\text{ for all } (t_0,s_0,r_0)\in\ctIII,
\end{align*}
where $\cttsr$ denotes the subtree of $\ctIII$ with the
root $(t_0,s_0,r_0)$.

As in \cite[Lemma~8.8]{BO10}, $(t,s,r)\in\ctIII$ implies
$(t,s)\in\ctII$ and $(s,r)\in\ctII$, so we can prove that
each $t$, $s$ or $r\in\ctI$ appears in not more than $C_{\rm sp}^2$
triples in $\ctIII$.
Denoting the depth of the cluster tree $\ctI$ by $p_\Idx$ and
using the standard estimate
\begin{equation*}
  \sum_{t\in\ctt} \#\hat t \leq (p_\Idx+1) \#\hat t_0
\end{equation*}
(cf., e.g., \cite[Corollary~3.10]{BO10}), we obtain the
following result:

%
%
\begin{theorem}[Complexity, multiplication]
\label{th:multiplication_complexity}
Let $p_\Idx$ denote the depth of the cluster tree $\ctI$.
We have
\begin{align*}
  W_{\rm mm}(t_0,s_0,r_0)
  &\leq C_{\rm mm} k^2 (p_\Idx + 1)
         (\#\hat t_0 + \#\hat s_0 + \#\hat r_0) &
  &\text{ for all } (t_0,s_0,r_0)\in\ctIII
\end{align*}
with $C_{\rm mm} := C_{\rm sp}^2 C_{\rm mb}$.
\end{theorem}

%
\paragraph{Forward substitution.}

Let us now consider the forward substitution, e.g., solving
\begin{equation*}
  L|_{\hat t\times\hat t} X|_{\hat t\times\hat s} = Y|_{\hat t\times\hat s}
\end{equation*}
for a block $(t,s)\in\ctII$.
We denote the number of operations required by $W_{\rm lfs}(t,s)$.

If $t$ is a leaf, we perform a forward substitution for
a standard matrix, which takes half the operations of the
matrix multiplication, so we certainly have
\begin{subequations}
\begin{equation}\label{eq:forward_leaf}
  W_{\rm lfs}(t,s) \leq W_{\rm mm}(t,t,s).
\end{equation}
If $t$ is not a leaf, we have $\sons(t)=\{t_1,t_2\}$ and our
algorithm performs the following steps
\begin{itemize}
  \item solve $L|_{\hat t_1\times\hat t_1} X|_{\hat t_1\times\hat s'}
        = Y|_{\hat t_1\times\hat s'}$ for all $s'\in\sons^+(s)$,
  \item compute $\widetilde Y|_{\hat t_2\times\hat s'}
        := Y|_{\hat t_2\times\hat s}
           - L|_{\hat t_2\times\hat t_1} X|_{\hat t_1\times\hat s'}$
        for all $s'\in\sons^+(s)$, and
  \item solve $L|_{\hat t_2\times\hat t_2} X|_{\hat t_2\times\hat s'}
        = \widetilde Y|_{\hat t_2\times\hat s'}$
        for all $s'\in\sons^+(s)$.
\end{itemize}
In total, we require
\begin{align}
  W_{\rm lfs}(t,s)
  &\leq \sum_{s'\in\sons^+(s)} W_{\rm lfs}(t_1,s')
        + W_{\rm mm}(t_2,t_1,s') + W_{\rm lfs}(t_2,s')
       \label{eq:forward_nonleaf}
\end{align}
\end{subequations}
operations.
Solving this recurrence relation yields the following estimate:

%
%
\begin{theorem}[Complexity, forward substitution]
\label{th:forward_complexity}
We have
\begin{align}
  W_{\rm lfs}(t,s) &\leq W_{\rm mm}(t,t,s) &
  &\text{ for all } (t,s)\in\ctII.\label{eq:forward_complexity}
\end{align}
\end{theorem}
\begin{proof}
By structural induction in $t\in\ctI$.

If $t\in\ctI$ is a leaf, (\ref{eq:forward_complexity}) follows
directly from (\ref{eq:forward_leaf}).

If  $t\in\ctI$ is not a leaf, let $\sons(t)=\{t_1,t_2\}$ and
assume that (\ref{eq:forward_complexity}) holds for $(t_1,s')$
and $(t_2,s')$ for all $s'\in\sons^+(s)$.
With (\ref{eq:forward_nonleaf}), we get
\begin{align*}
  W_{\rm lfs} &(t,s)
   \leq \sum_{s'\in\sons^+(t)} W_{\rm lfs}(t_1,s')
                + W_{\rm mm}(t_2,t_1,s')\\
  &\qquad + W_{\rm lfs}(t_2,s')\\
  &\leq \sum_{s'\in\sons^+(t)} \kern -8pt W_{\rm mm}(t_1,t_1,s')
                + W_{\rm mm}(t_2,t_1,s')\\
  &\qquad + W_{\rm mm}(t_2,t_2,s')\\
  &\leq \sum_{s'\in\sons^+(t)} \kern -8pt W_{\rm mm}(t_1,t_1,s')
            + W_{\rm mm}(t_2,t_1,s')\\
  &\qquad + W_{\rm mm}(t_1,t_2,s')
            + W_{\rm mm}(t_2,t_2,s')\\
  &= \sum_{t'\in\sons^+(t)} \sum_{t''\in\sons^+(t)} \sum_{s'\in\sons^+(s)}
            W_{\rm mm}(t',t'',s')\\
  &\leq W_{\rm mm}(t,t,s).
\end{align*}
This completes the induction.
\qed
\end{proof}

Our algorithm for solving $X|_{\hat t\times\hat s} R|_{\hat s\times\hat s}
= Y|_{\hat t\times\hat s}$ can be treated by the same arguments to find
that the number of operations $W_{\rm rfs}(t,s)$ is bounded by $W_{\rm mm}(t,s,s)$.

%
\paragraph{LR factorization.}

Finally we consider the LR factorization, e.g., finding
$L|_{\hat t\times\hat t},R|_{\hat t\times\hat t}$ such that
\begin{equation*}
  L|_{\hat t\times\hat t} R|_{\hat t\times\hat t} = A|_{\hat t\times\hat t}.
\end{equation*}
Let $W_{\rm lr}(t)$ denote the number of required operations.

If $t$ is a leaf, computing the LR factorization directly
requires one third of the number of operations required to
compute the matrix product, so we certainly have
\begin{subequations}
\begin{equation}\label{eq:lr_leaf}
  W_{\rm lr}(t) \leq W_{\rm mm}(t,t,t).
\end{equation}
If $t$ is not a leaf, i.e., if $\sons(t)=\{t_1,t_2\}$ holds, our
recursive algorithm performs the following steps:
\begin{itemize}
  \item compute the factorization of $A|_{\hat t_1\times\hat t_1}$,
  \item solve $L|_{\hat t_1\times\hat t_1} R|_{\hat t_1\times\hat t_2} =
    A|_{\hat t_1\times\hat t_2}$ for $R|_{\hat t_1\times\hat t_2}$,
  \item solve $L|_{\hat t_2\times\hat t_1} R|_{\hat t_1\times\hat t_1} =
    A|_{\hat t_2\times\hat t_1}$ for $L|_{\hat t_2\times\hat t_1}$,
  \item compute $\widetilde A|_{\hat t_2\times\hat t_2}
    = A|_{\hat t_2\times\hat t_2} - L|_{\hat t_2\times\hat t_1}
      R|_{\hat t_1\times\hat t_2}$,
  \item compute the factorization of $\widetilde A|_{\hat t_2\times\hat t_2}$,
\end{itemize}
and this takes
\begin{align}
  W_{\rm lr} (t)
  &= W_{\rm lr}(t_1) + W_{\rm lfs}(t_1,t_2) + W_{\rm rfs}(t_2,t_1)
     + W_{\rm mm}(t_2,t_1,t_2) + W_{\rm lr}(t_2)\notag\\
  &\leq W_{\rm lr}(t_1) + W_{\rm mm}(t_1,t_1,t_2) + W_{\rm mm}(t_2,t_1,t_1)
     + W_{\rm mm}(t_2,t_1,t_2) + W_{\rm lr}(t_2)\label{eq:lr_nonleaf}
\end{align}
\end{subequations}
operations due to Theorem~\ref{th:forward_complexity}.
We can solve this recurrence relation to obtain the following
result:

%
%
\begin{theorem}[Complexity, LR]
\label{th:lr_complexity}
We have
\begin{align}\label{eq:lr_complexity}
  W_{\rm lr}(t) &\leq W_{\rm mm}(t,t,t) &
  &\text{ for all } t\in\ctI.
\end{align}
\end{theorem}
\begin{proof}
By structural induction in $t\in\ctI$.

If $t\in\ctI$ is a leaf, (\ref{eq:lr_complexity}) follows
directly from (\ref{eq:lr_leaf}).

If $t\in\ctI$ is not a leaf, we have $\sons(t)=\{t_1,t_2\}$ and
assume that (\ref{eq:lr_complexity}) holds for $t_1$ and $t_2$.
With (\ref{eq:lr_nonleaf}), we get
\begin{align*}
  W_{\rm lr} (t)
  &\leq W_{\rm lr}(t_1) + W_{\rm mm}(t_1,t_1,t_2)\\
  &\qquad + W_{\rm mm}(t_2,t_1,t_1) + W_{\rm mm}(t_2,t_1,t_2) + W_{\rm lr}(t_2)\\
  &\leq W_{\rm mm}(t_1,t_1,t_1) + W_{\rm mm}(t_1,t_1,t_2)\\
  &\qquad + W_{\rm mm}(t_2,t_1,t_1) + W_{\rm mm}(t_2,t_1,t_2)
          + W_{\rm mm}(t_2,t_2,t_2)\\
  &\leq \sum_{\substack{t'\in\sons^+(t),\\
                        t''\in\sons^+(t),\\
                        t'''\in\sons^+(t)}}
         W_{\rm mm}(t',t'',t''')
   \leq W_{\rm mm}(t,t,t).
\end{align*}
This completes the induction.
\qed
\end{proof}

Now we can state the final result for the complexity of the
approximation of the LR factorization:

%
%
\begin{corollary}[Total complexity]
Our algorithm constructs the approximate LR factorization of
local rank $k$ in ${\mathcal O}(k^2 (p_\Idx+1) n_\Idx)$ operations, where
$p_\Idx$ denotes the depth of the cluster tree.
\end{corollary}
\begin{proof}
Combine Theorem~\ref{th:multiplication_complexity} with
Theorem~\ref{th:lr_complexity} and apply it to $t=\treeroot(\ctI)$.
\qed
\end{proof}

%
%
\begin{remark}[Inversion]
\label{rm:inversion}
By a similar approach, we can also obtain an approximation of the
inverse:
consider $A|_{\hat t\times\hat t}^{-1}$ for $t\in\ctI$.
If $t$ is a leaf, we can compute the inverse directly.
Otherwise we let
\begin{equation*}
  A|_{\hat t\times\hat t}
  = \begin{pmatrix}
    A_{11} & A_{12}\\
    A_{21} & A_{22}
  \end{pmatrix}
\end{equation*}
as before and perform a block Gauss elimination to obtain
\begin{equation*}
  A|_{\hat t\times\hat t}
  = \begin{pmatrix}
    I & \\
    A_{21} A_{11}^{-1} & I
  \end{pmatrix}
  \begin{pmatrix}
    A_{11} & A_{12}\\
    & A_{22} - A_{21} A_{11}^{-1} A_{12}
  \end{pmatrix}.
\end{equation*}
We denote the Schur complement by $S := A_{22} - A_{21} A_{11}^{-1} A_{12}$,
compute the inverses of the two block triangular matrices, and find
the representation
\begin{equation*}
  C := A|_{\hat t\times\hat t}^{-1}
  = \begin{pmatrix}
    A_{11}^{-1} & -A_{11}^{-1} A_{12} S^{-1}\\
    & S^{-1}
  \end{pmatrix}
  \begin{pmatrix}
    I & \\
    -A_{21} A_{11}^{-1} & I
  \end{pmatrix}
\end{equation*}
of the inverse.
This equation allows us to compute the inverse $A|_{\hat t\times\hat t}^{-1}$
by recursively computing $A_{11}^{-1}$ and $S^{-1}$ and carrying out
six matrix multiplications:
we start with $A_{11}^{-1}$, compute $B_{12} := A_{11}^{-1} A_{12}$,
$B_{21} := A_{21} A_{11}^{-1}$ and $S = A_{11} - A_{21} B_{12}$,
determine $C_{22} = S^{-1}$ by recursion, and finish by computing
$C_{12} = -A_{11}^{-1} A_{12} S^{-1} = -B_{12} S^{-1}$,
$C_{21} = -S^{-1} A_{21} A_{11}^{-1} = -S^{-1} B_{21}$ and
$C_{11} = A_{11}^{-1} + A_{11}^{-1} A_{12} S^{-1} A_{21} A_{11}^{-1}
 = A_{11}^{-1} - C_{12} B_{21}$.
Following the same reasoning as before, we can prove that the
number of operations $W_{\rm inv}(t)$ is bounded by
$W_{\rm mm}(t,t,t)$, and Theorem~\ref{th:multiplication_complexity}
yields that ${\mathcal O}(k^2 (p_\Idx+1) n_\Idx)$ operations are
sufficient to construct the approximate inverse.
\end{remark}

\section{Numerical experiments}

We investigate the practical properties of the new algorithms by
considering two standard model problems:
for the first model problem, we consider the linear system
resulting from a finite element discretization of Poisson's
equation on the unit square using piecewise linear nodal basis functions
on a regular mesh.
We use a domain decomposition cluster strategy similar to the
one described in \cite{GRKRLE05a} to find a suitable cluster tree
and block tree for the sparse system.
Since the matrix is symmetric and positive definite, we construct
an approximate Cholesky factorization by the algorithms described
in Section~\ref{se:algebraic_operations} using a block-relative
accuracy of $\hat\epsilon\in\bbbr_{>0}$ for the recompression.

%
%
\begin{table}
  \begin{equation*}
    \begin{array}{rr|rr|rr|rrr}
      \multicolumn{2}{c|}{\text{Grid}} &
      \multicolumn{2}{c|}{\text{Param}} &
      \multicolumn{2}{c|}{\text{Setup}} &
      \multicolumn{3}{c}{\text{Solve}}\\
      \ell & n & \eta & \hat\epsilon & \text{Time}/n & \text{Mem}/n
      &\text{Err} & m & \text{Time/n} \\
      \hline
      7 &      16\,129 & 4 & 3.1_{-3} & 7.0_{-5} & 1.0 & 0.06 & 3 & 1.9_{-6}\\
      8 &      65\,025 & 4 & 7.7_{-4} & 8.4_{-5} & 1.1 & 0.07 & 3 & 2.3_{-6}\\
      9 &     261\,121 & 4 & 1.9_{-4} & 1.1_{-4} & 1.2 & 0.07 & 3 & 2.6_{-6}\\
     10 &  1\,046\,529 & 4 & 4.8_{-5} & 1.3_{-4} & 1.2 & 0.10 & 3 & 2.8_{-6}\\
     11 &  4\,190\,209 & 4 & 1.2_{-5} & 1.5_{-4} & 1.2 & 0.11 & 3 & 3.0_{-6}\\
     12 & 16\,769\,025 & 4 & 3.0_{-6} & 1.7_{-4} & 1.2 & 0.10 & 3 & 3.5_{-6}
    \end{array}
  \end{equation*}
  \caption{Preconditioner for the finite element model problem.}
  \label{ta:fem2d}
\end{table}

Table~\ref{ta:fem2d} lists the results for grid levels $7$ to $12$
with $16\,129$ to $16\,769\,025$ grid points.
The column ``Param'' gives the admissibility parameter $\eta$ and
the relative accuracy $\hat\epsilon$ used in the recompression
algorithm, the column ``Setup'' gives the time (in seconds per degree
of freedom) and storage requirements (in KB per degree of freedom) for
constructing the Cholesky factorization, and the column ``Solve''
gives an approximation of the convergence factor $\|I-\widetilde A^{-1} A\|_2$
obtained by the power iteration, the number $m$ of iteration
steps required by the conjugate gradient iteration to reduce the
relative residual norm below $10^{-8}$, and the time required
per step and degree of freedom.
The experiments were carried out on a single core of an AMD Opteron 8431
processor.

We can see that choosing the accuracy $\hat\epsilon$ like ${\mathcal O}(h^2)$
to keep up with the ${\mathcal O}(h^{-2})$ growth of the condition number
is sufficient to obtain a stable convergence rate.
The time per degree of freedom grows like $\log(n)$, the storage
requirements are bounded.
This suggests that, surprisingly, our strategy for controlling the
accuracy leads to ranks that are constant on average.

For the second model problem, we consider the linear system resulting
from the finite element discretization of the single layer potential
operator on the two-dimensional unit circle using a regular mesh and
piecewise constant basis functions.

%
%
\begin{table}
  \begin{equation*}
    \begin{array}{rr|rr|rr|rrr}
      \multicolumn{2}{c|}{\text{Grid}} &
      \multicolumn{2}{c|}{\text{Param}} &
      \multicolumn{2}{c|}{\text{Setup}} &
      \multicolumn{3}{c}{\text{Solve}}\\
      \ell & n & \eta & \hat\epsilon & \text{Time}/n & \text{Mem}/n
      &\text{Err} & m & \text{Time/n} \\
      \hline
     11 &      8\,192 & 2 & 1.2_{-4} & 1.6_{-4} & 0.8 & 0.05 & 3 & 1.6_{-6}\\
     12 &     16\,384 & 2 & 6.1_{-5} & 1.9_{-4} & 0.8 & 0.03 & 3 & 1.6_{-6}\\
     13 &     32\,768 & 2 & 3.1_{-5} & 2.1_{-4} & 0.8 & 0.03 & 3 & 1.7_{-6}\\
     14 &     65\,536 & 2 & 1.5_{-5} & 2.4_{-4} & 0.8 & 0.02 & 3 & 1.8_{-6}\\
     15 &    131\,072 & 2 & 7.6_{-6} & 2.6_{-4} & 0.8 & 0.02 & 3 & 1.9_{-6}\\
     16 &    262\,144 & 2 & 3.8_{-6} & 3.0_{-4} & 0.8 & 0.02 & 3 & 2.0_{-6}\\
     17 &    524\,288 & 2 & 1.9_{-6} & 3.3_{-4} & 0.8 & 0.02 & 3 & 2.1_{-6}\\
     18 & 1\,048\,576 & 2 & 9.5_{-7} & 3.5_{-4} & 0.8 & 0.01 & 4 & 2.2_{-6}\\
     19 & 2\,097\,152 & 2 & 4.8_{-7} & 3.6_{-4} & 0.8 & 0.02 & 3 & 2.2_{-6}\\
     20 & 4\,194\,304 & 2 & 2.4_{-7} & 3.8_{-4} & 0.8 & 0.20 & 4 & 2.2_{-6}
    \end{array}
  \end{equation*}
  \caption{Preconditioner for the boundary element model problem.}
  \label{ta:bem2d}
\end{table}

Table~\ref{ta:bem2d} lists the results for this model problem.
Choosing $\hat\epsilon$ like ${\mathcal O}(h)$ compensates for the
${\mathcal O}(h^{-1})$ growth of the condition number and leads
to stable convergence rates.
As in the case of the partial differential equation, the time per
degree of freedom grows like $\log(n)$ and the storage requirements
per degree of freedom are bounded.
This suggests that also in this case our algorithm chooses ranks
that are constant on average.

We can conclude that the new preconditioner requires a setup time
of ${\mathcal O}(n \log n)$ and ${\mathcal O}(n)$ units of storage
to ensure stable $h$-independent convergence of the conjugate
gradient method.

\bibliographystyle{plain}
\bibliography{scicomp}

\end{document}